\documentclass[11pt,twoside, reqno]{amsart}

\usepackage{amsmath}
\usepackage{amsthm}
\usepackage{amsfonts, amssymb}
\usepackage{mathrsfs}
\usepackage[all]{xy}
\usepackage{url}

\usepackage{latexsym}
\usepackage[dvips]{graphicx}

\theoremstyle{plain}
\newtheorem{lema}{Lemma}
\newtheorem{prop}[lema]{Proposition}
\newtheorem{teo}[lema]{Theorem}

\newtheorem{coro}[lema]{Corollary}
\newtheorem{corollary}[lema]{Corollary}
\newtheorem{remark}[lema]{Remark}

\newtheorem{obs}[lema]{Remark}
\theoremstyle{definition}
\newtheorem{defi}[lema]{Definition}
\newtheorem{definition}[lema]{Definition}
\newtheorem{ej}[lema]{Example}

\newcommand{\N}{\mathbb{N}}

\newcommand{\E}{\mathbb{E}}

\newcommand{\TC}{\textrm{TC}}

\newcommand{\st}{\textrm{st}}

\newcommand{\cat}{{\sf{cat}}}
\newcommand{\Cat}{{\sf{Cat}}}
\newcommand{\tc}{{\sf{TC}}}
\newcommand{\pr}{{\rm {pr}}}

\DeclareMathOperator{\join}{\circledast}
\DeclareMathOperator{\joinp}{\ast}

\begin{document}

\title{Topology of medial regime random simplicial complexes }

\author[J.A. Barmak]{Jonathan Ariel Barmak$^{\dagger}$}

\thanks{$^{\dagger}$ Researcher of CONICET. Partially supported by grants PUE IMAS 2017, PICT-
2019-02338 and UBACyT 20020190100099BA}

\address{Universidad de Buenos Aires. Facultad de Ciencias Exactas y Naturales. Departamento de Matem\'atica. Buenos Aires, Argentina.}

\email{jbarmak@dm.uba.ar}
\author[M. Farber]{Michael Farber$^{\dagger\dagger}$}

\thanks{$^{\dagger\dagger}$ M. Farber was partially supported by a DMS-EPSRC research grant}

\address{School of Mathematical Sciences, Queen Mary University of London, London, UK}

\email{M.Farber@qmul.ac.uk}


\begin{abstract} 
We analyse topology of random simplicial complexes in the medial regime.
We show that these complexes are highly connected and have homotopy type of iterated suspensions. 
One of our main tools is a new combinatorial criterion for high connectivity of simplicial complexes, which is more flexible than conicity. 
We show that topological complexity of random simplical complexes in the medial regime is bounded above by $2$ and it equals $2$ for a class of homogenous medial regime random simplicial complexes, a.a.s. 
\end{abstract}


\maketitle

\section{Introduction}
Random simplicial complexes are used in various branches of mathematics and in other sciences. Most recently, random simplicial complexes appear as a major tool to model large complex networks consisting of a vast number of interacting objects. While pairwise interactions
can be modelled by a graph, the higher order interactions between the objects
require the language of simplicial complexes, see \cite{Bat}.
Mathematical theory of random simplicial complexes is an active research field, we can refer to a recent survey \cite{BK}. 

In the mathematical literature there exist several probabilistic models of simplicial complexes. 
In this paper we shall be mainly working with the multi-parameter model, which was developed and 
analysed in \cite{CF}, \cite{FMN}. The multi-parameter model includes, as special cases, the Linial -- Meshulam model \cite{LM} and some other well-known models.

The principal question we tackle in this paper concerns the topological complexity of random simplicial complexes. 
The concept
of topological complexity,  $\TC(X)$,
 introduced in \cite{Far0}, is motivated by problems of engineering and robotics. $\TC(X)$ is a numerical invariant associated with every path-connected topological space $X$. 
 The value of $\TC(X)$ is
 a reflection of navigational complexity of motion algorithms for autonomous systems having $X$ as their configuration space, see \cite{Far0}. 
It is well-known that the numerical value of $\TC(X)$ can be arbitrarily large, see for example Theorem 4.57 in \cite{Far2}. However, surprisingly, as we shall establish in this paper, the value $\TC(X)$ is at most $2$ for random simplicial complexes in the medial regime. Moreover, we show that $\tc(X)=2$, a.a.s., for a special class of homogeneous medial regime random simplicial complexes. 


The major results of this paper are stated as Theorem \ref{thm:tcle2} and Theorem \ref{thm:tc2}. 
These results improve Theorem 8.1 from \cite{Far}, stating that topological complexity $\tc(X)$ of a random simplicial complex $X$ in the medial regime is bounded above by 4, with probability tending to one, when the number $n$ of vertices tends to infinity. 
The proof of Theorem 8.1 in \cite{Far} uses a bound for the dimension of a random complex \cite{FM}, a bound for the connectivity of a random complex obtained using the notion of conicity of simplicial complexes \cite{Bar, Bar2}, and the upper bound for $\TC(X)$ involving the dimension and the connectivity, see Theorem 4.16 from \cite{Far2}. 


The progress made in this paper is based on two achievements. Firstly, we introduce and analyse a new combinatorial property $Q_r$ which is weaker than conicity but still implies connectivity. Using property $Q_r$ we show that the Lusternik - Schnirelmann category $\cat(X)$ 
of a random simplicial complex $X$ in the medial regime satisfies $\cat(X)\le 1$, see Theorem \ref{thm:thm1}. This leads to Theorem \ref{thm:tcle2} stating that $\TC(X)\le 2$, a.a.s. 

The second main achievement of this paper is in establishing the opposite inequality $\TC(X)\ge 2$ for a random simplicial complex $X$. 
In this part of the paper we restrict our attention to  {\it a homogeneous medial regime} (HMR) random simplicial complex $X$, see \S \ref{sec:hmr}, which is a special class of a medial regime random complex. 
Theorem \ref{thm:tc2} is based on the analysis of the Betti numbers of a HMR random complex. We show that the expected values of the Betti numbers form a {\it unimodal} sequence, i.e. this sequence is monotone increasing until it reaches the maximum and then it starts its monotone decrease. The dimension of the maximum of the expected Betti numbers is {\it the critical dimension}. We use concentration inequalities to show that the Betti number in the critical dimension tends to infinity. 
Theorem \ref{thm:tc2} states that {\it for a subset $\N'\subset \N$ of density 1, the probability that an HMR random simplicial complex $X\subset \Delta_n$ satisfies $\tc(X)=2$ tends to one as $n\to \infty$, $n\in \N'$. } Here $\Delta_n$ denotes the simplex with the vertex set $[n]= \{1, 2, \dots, n\}$. 
 

At the moment it is not clear whether the statement of Theorem \ref{thm:tc2} becomes false with $\N'=\N$, or, rather, a generalisation of our result can be achieved by a different method. The number of vertexes is a natural parameter of the problem and it plays an important role in our approach. 

The limit case $n=\infty$ of random simplicial complexes on a countable set of vertexes was studied in \cite{Far}, and the situation there is totally opposite:  a random complex $X\subset \Delta_{\N}$ on countably many vertexes is isomorphic to {\it the Rado complex} \cite{EFM} and its geometric realisation is homeomorphic to the infinite-dimensional simplex with probability one, see \cite{EFM}. Hence, in this case 
$\TC(X)=0$ with probability one, see Theorem 11.2 and Theorem 12.1 from \cite{Far}. 
This connection with the Rado complex \cite{EFM} may serve as {\it \lq\lq an intuitive explanation\rq\rq}\, of our main results concerning the topological complexity of random simplicial complexes in the medial regime as these finite simplicial complexes can be viewed as the induced subcomplexes of the Rado complex on a randomly chosen large number of vertexes. 

The authors are thankful to John Oprea for helpful comments.


%
%


\section{Property $Q_r$ and high connectivity}

\subsection{Connectivity and conicity} 

Recall that a topological space $X$ is called {\it $r$-connected} (where $r\ge 0$ is an integer) if every continuous map $S^d\to X$ from the $d$-dimensional sphere into $X$, with $d\le r$, can be extended to a continuous map of the disk $D^{d+1}\to X$. In addition, we shall say that a space $X$ is {\it $(-1)$-connected} if it is non-empty, $X\not=\emptyset$.

The following combinatorial notion was introduced in \cite{Bar}: a simplicial complex $K$ is called \textit{$r$-conic} (where $r\ge 1$) if every subcomplex $L\subseteq K$ with at most $r$ vertices is contained in a simplicial cone, or, equivalently, every subcomplex $L\subseteq K$ with at most $r$ vertices
is contained in the closed star $\st_K (v)$ of a vertex $v\in K$. 





\begin{teo} For $r\ge 0$ any $(2r+2)$-conic simplicial complex $K$ is $r$-connected. 
\end{teo}

The above result appears in \cite{Bar2}; its proof 
uses the Nerve lemma \cite{Bjo}, which we now recall.

\begin{teo}[Nerve lemma] \label{nerv}
Let $K$ be a simplicial complex and let $\{L_i\}_{i\in I}$ be a family of subcomplexes covering $K$. If, for some $r\ge 0$, each non-empty intersection 
$$L_{i_1}\cap L_{i_2} \cap \ldots \cap L_{i_t}\quad \mbox{with}\quad  1\le t\le r+1$$ is $(r+1-t)$-connected, then the complex $K$ is $r$-connected if and only if the nerve $\mathcal{N}(\{L_i\}_{i\in I})$ is $r$-connected.
\end{teo}  

  \subsection{The property $Q_r$}









The following definition is slightly weaker than $(2r+2)$-conicity but it still implies $r$-connectivity, as we shall show below, see Theorem \ref{thm:6}. 


\begin{defi}\label{def:qr}{\it 
We say that a simplicial complex $K$ has property $Q_r$ (where $r\ge 0$ is an integer)
if every subcomplex $L\subseteq K$ with at most $2r+2$ vertices and at most $3^{r+1}-1$ simplexes is contained in the closed star $\st_K(v)$ of a vertex $v\in K$. }
\end{defi}
\begin{remark}\label{rk1} {\rm The following comments are related to Definition \ref{def:qr}:}
\begin{enumerate}
{\rm 
\item[{(1)}] 
A simplicial complex with at most $2r+2$ vertices may potentially have 
$2^{2r+2}-1=4^{r+1}-1$ simplexes, and hence the bound $3^{r+1}-1$ in Definition \ref{def:qr} is meaningful. 


\item[{(2)}] A simplicial complex $K$ satisfies $Q_0$ if and only if every two vertices $v_1, v_2$ of $K$ are 
either connected by an edge, or $v_1$ and $v_2$ have a common neighbour, i.e. for some vertex $v\in K$ the edges $vv_1$ and $vv_2$ are in $K$. 
Thus, clearly, the property $Q_0$ implies connectivity of the complex $K$. 

\item[{(3)}] Obviously, any $(2r+2)$-conic complex satisfies $Q_r$. 

\item[{(4)}] Consider the complex $K$, which is the boundary of a 3-dimensional simplex. 
The complex $K$ has 4 vertexes and is not $4$-conic, as $\st_K(v)\not= K$ for any vertex $v\in K$. 
The number of simplexes of $K$ equals $2^4-2=14$, and 
any subcomplex $L\subseteq K$ with at most $3^2-1=8$ simplexes must be proper and must not include one of the 2-dimensional faces of $K$. Therefore,
 $L$ must be contained in the closed star of a vertex of $K$. We see that $K$
 satisfies $Q_1$. 
 We conclude that  {\it in general, property
$Q_r$ can be weaker than $(2r+2)$-conicity. }

\item[{(5)}] One may rephrase property $Q_r$ as follows. We say that a simplicial complex $L$ is starlike if for a vertex $v\in L$ one has 
$L= \st_L(v)$. This means that every simplex of $L$ either contains $v$ or is a face of a simplex containing $v$. 
It is obvious that in Definition \ref{def:qr} one may disregard the starlike simplicial complexes $L$. Thus, a simplicial complex $K$ has property $Q_r$ 
iff every non-starlike subcomplex $L\subseteq K$ with at most $2r+2$ vertices and at most $3^{r+1}-1$ simplexes is contained in the closed star $\st_K(v)$ of a vertex $v\in K$.
}
\end{enumerate}
\end{remark}
\subsection{The operation of amalgamated join} 
Let $K$ and $L$ be simplicial complexes with their vertex sets not necessarily disjoint. We denote by $K \circledast L$ the \textit{amalgamated join} of the complexes $K$ and $L$. The simplexes of the simplicial complex $K \circledast L$ are the unions $\sigma \cup \tau$ with $\sigma \in K$, $\tau \in L$, and also the simplexes of $K$ and the simplexes of $L$. 

If the vertex sets of $K$ and $L$ are disjoint, then the complex $K \circledast L$ coincides with the usual join $K *L$. 
In general, the complex $K\circledast L$ is a quotient of $K*L$. Note that the operation $\circledast$ is commutative and associative.

\begin{ej}\label{ex4}
Let $K$ be a simplicial complex, let $L\subseteq K$ be a subcomplex and let $v$ be a vertex of $K$. Then $L \subseteq \st_K (v)$ if and only if the cone $v\circledast L $ is contained in $K$. Moreover, if $v_1, v_2, \ldots, v_t$ are $t$ vertices of $K$, then $$L\subseteq \bigcap \st_K(v_i)$$ if and only if $$Z \circledast L \subseteq K,$$ where $Z=\{v_1,v_2,\ldots, v_t\}\subseteq K$ denotes the 0-dimensional subcomplex.
\end{ej}

\begin{defi}{\it 
Let $K$ be a simplicial complex, and let $L\subseteq K$ be a subcomplex. The \textit{star} $\st_K(L)$ of $L$ in $K$ is the subcomplex given by those simplexes $\sigma \in K$ such that $\sigma \circledast L\subseteq K$. In other words, it is the biggest subcomplex $L'$ of $K$ such that $L'\circledast L\subseteq K$. 
}\end{defi}

This definition coincides with the usual definition of a closed star in the special case when $L$ is a vertex or a simplex. 

\begin{remark}{\rm 
(1) Note that $\st_K (L)$ is non-empty if and only if $L$ is contained in the star of a vertex of $K$. Thus, a complex $K$ is $r$-conic if and only if the star of every subcomplex $L\subseteq K$ of at most $r$ vertices is non-empty. 

(2) $\st_K(\emptyset)=K$. 

(3) If $Z$ is a $0$-dimensional complex, $\st_K(Z)$ is the intersection of the closed stars of its vertices.}
\end{remark}

\subsection{The connectivity theorem} 
\begin{teo} \label{thm:6} Any simplicial complex with property $Q_r$, where $r\ge 0$, is $r$-connected.
\end{teo}
\begin{proof}
The statement of Theorem \ref{thm:6} is trivial for $r=0$, see Remark \ref{rk1}, part (2). 
This case will be used as the base of induction. 

Assuming that $K$ has property $Q_r$, where $r\ge 1$, consider the cover 
$$
\mathcal{U}=\{\st_K(v)\}_{v\in K}$$
of $K$ formed by the stars of vertices, and apply the Nerve Lemma (see Theorem \ref{nerv}). 
We claim that the nerve $\mathcal N(\mathcal{U})$ of this cover {\it has the complete skeleton up to the dimension} $2r+1$ and, 
since $r<2r+1$, 
the nerve $\mathcal N(\mathcal{U})$ is $r$-connected.
Indeed, if $v_1, v_2, \ldots, v_t$ are pairwise distinct vertices of $K$, where $t\le 2r+2$, we may apply the property $Q_r$ to the 0-dimensional subcomplex 
$L=\{v_1, v_2, \ldots, v_t\}$ and conclude that $L\subseteq \st_K(v)$ for a vertex $v\in K$. Here we use the inequalities
$t\le 2r+2\le 3^{r+1}-1.$
Therefore, $v\in  \bigcap \st_K(v_i),$ implying that $\{v_1, v_2, \dots, v_t\}$
is a simplex of the nerve complex 
$\mathcal N(\mathcal{U})$.

To be able to apply Theorem \ref{nerv} we need to show that for any set of vertices $v_1,v_2,\ldots, v_t \in K$, where 
$1\le t\le r+1$, the complex 
$$L=\bigcap_{i=1}^t \st_K(v_i)$$
is $(r-t+1)$-connected. This is trivial  for $t=1$ as the stars are contractible. 
For $t\ge 2$, we may use induction, and thus it suffices to prove that the complex $L$ satisfies $Q_{r-t+1}$. 

Let $L'$ be a subcomplex of $L$ with at most $$2(r-t+1)+2=2r-2t+4$$ vertices and at most $3^{r-t+2}-1$ simplexes. Let $Z\subseteq K$ denote the 0-dimensional subcomplex of $K$ with vertices $v_1,v_2,\ldots, v_t$. Since $L'\subseteq L=\st_K(Z)$, the complex $L'\join Z$ is a subcomplex of $K$ (see Example \ref{ex4}). Moreover, the number of vertices of $L'\join Z$ is at most $$2r-2t+4+t=2r-t+4\le 2r+2,$$ and the number of simplexes of $L'\join Z$
is at most $$(3^{r-t+2}-1+1)(t+1)-1=3^{r+1}3^{1-t}(t+1)-1\le 3^{r+1}-1,$$ since $t\ge 2$. By induction hypothesis, using the property $Q_r$ of $K$, we see that the complex $L'\join Z$ is contained in the star $\st_K(v)$ of a vertex $v\in K$. In other words, $L'\join Z\join v \subseteq K$ and thus 
$L'\join v \subseteq \st_K(Z)=L$, which means that $v\in L$ and $L' \subseteq \st_L(v)$. We see that the complex $L$ has property $Q_{r-t+1}$ and hence it is $(r-t+1)$-connected by induction. 

The Nerve Lemma is now applicable and yields the desired result. 
\end{proof}

\section{Random simplicial complexes and the property $Q_r$}\label{sec:3}

\subsection{Random simplicial complexes}\label{sec:31} In this paper we consider the lower multi-para\-meter model of random simplexes described in \cite{FMN} which we shall now briefly recall.
The symbol 
$\Delta_n$ denotes the standard simplex on the vertex set $[n]=\{1, 2, \dots, n\}$ and a random simplicial complex in this model is a probability measure 
$\mathcal P$ on the set of all simplicial subcomplexes of $\Delta_n$. The measure 
$\mathcal P$ depends on a system of probability parameters 
\begin{eqnarray*}\label{psigma0}
p_\sigma\in [0,1]
\end{eqnarray*} 
associated with each simplex $\sigma$ of $\Delta_n$. 
The probability measure $\mathcal P$ is uniquely characterised by the property that for every simplicial complex $L\subset \Delta_n$ the probability that a random complex $X\subset \Delta_n$ contains $L$ equals 
\begin{eqnarray}\label{eq:x}
\mathcal P\{X; L\subset X\} = \prod_{\sigma\in L} p_\sigma;
\end{eqnarray}
we refer to Corollary 5.3 from \cite{FMN}. 

In this section we prove the following result. 

\begin{teo}\label{thm8}
Suppose that all probability parameters $p_\sigma$ satisfy 
\begin{eqnarray}\label{psigma}
0<p\le p_\sigma, 
\end{eqnarray}
where $p>0$ is a constant independent of $n$. Then for any integer $r=r(n)\ge 0$ satisfying
\begin{eqnarray}\label{eq:5}
r+1 \le \log_3\log_{q^{2}}n, \quad \mbox{where}\quad q=p^{-1},
\end{eqnarray}
a random simplicial complex $X\subset \Delta_n$ in the lower multi-parameter model possesses property $Q_r$, a.a.s. 
In other words, under conditions (\ref{psigma}) and (\ref{eq:5}) the $\mathcal P$-measure of the set of simplicial complexes $X\subset \Delta_n$ having property $Q_r$ tends to $1$ as $n$ tends to $\infty$.
\end{teo}

\begin{proof} We estimate above the probability that a random simplicial complex $X\subset \Delta_n$ does not have property $Q_r$.

We shall denote by $V_L$ the set of all vertexes of a simplicial complex $L$. Besides, the symbol $F(L)$ will denote the set of all simplexes of $L$. 

For a non-starlike simplicial complex $L\subset \Delta_n$ (see Remark \ref{rk1}, part (5)) we denote by $\mathcal C_L$ the set of all simplicial complexes $X \subset \Delta_n$ containing $L$; the probability 
of $\mathcal C_L$ is given by (\ref{eq:x}). For a vertex $v\in [n]-V_L$ we have $L\subset L\cup v\subset L\ast v$ and 
$ {\mathcal C}_{L\ast v}\subset \mathcal C_{L\cup v}\subset {\mathcal C}_L$. 
Any simplicial complex $$X\in \mathcal C_{L}- \mathcal C_{L\ast v}$$ contains $L$ as a subcomplex but $L\ast v\not\subset X$, i.e. either $v\notin X$, or $v\in X$ and $L\not\subset \st_X(v)$. 
Thus, we see, that the set ${\mathcal {NQ}}_r$ of all simplicial complexes $X\subset \Delta_n$ which do not possess property $Q_r$ satisfies
\begin{eqnarray*}\label{eq:noqr}
{\mathcal {NQ}}_r \, 
\subset  \, 
\bigcup_L \bigcap_{v\in [n]-V_L}\left(\mathcal C_{L}-\mathcal C_{L\ast v}\right),
\end{eqnarray*}
where $L\subset \Delta_n$ runs over all not-starlike simplicial subcomplexes with at most $ 2r+2$ vertexes and at most $3^{r+1}-1$ simplexes. 
Hence we have
\begin{eqnarray}\label{eq:7}
\mathcal P(\mathcal {NQ}_r) &\le& 
\sum_L \, \mathcal P\left(\bigcap_{v\in [n]-V_L}(\mathcal C_{L}-\mathcal C_{L\ast v})\right) \nonumber\\ 
&=&
\sum_L \, \mathcal P\left(\left(\bigcap_{v\in [n]-V_L}(\mathcal C_{L}-\mathcal C_{L\ast v})\right)\, \big|\, \mathcal C_L\right)\cdot \mathcal P(\mathcal C_L).
\end{eqnarray}
The symbol $$\mathcal P\left(\left(\bigcap_{v\in [n]-V_L}(\mathcal C_{L}-\mathcal C_{L\ast v})\right)\, |\, \mathcal C_L\right)$$ denotes the conditional probability
of the event 
$\bigcap_{v\in [n]-V_L}(\mathcal C_{L}-\mathcal C_{L\ast v})$ assuming $\mathcal C_L$. 

The conditional probability $\mathcal P(\mathcal C_{L\ast v} \, |\, \mathcal C_L)$ equals 
$
p_v \prod_{\sigma\in L}p_{v\sigma}
$
as follows directly from formula (\ref{eq:x}). 
Moreover, we claim that 
\begin{eqnarray}\label{eq:cond}
\mathcal P\left(\left(\bigcap_{v\in [n]-V_L}(\mathcal C_{L}-\mathcal C_{L\ast v})\right)\, |\, \,  \mathcal C_L\right)= \prod_{v\in [n]-V_L}\left( 1- p_v\prod_{\sigma\in L}p_{\sigma v}\right).
\end{eqnarray}
To prove this we observe that the family of events $\{\mathcal C_{L\ast v}\}_{v\in [n]-V_L}$ is collectively conditionally independent 
(see \cite{Shi}, p. 36) over $\mathcal C_L$. 
Indeed, for every subset $J\subset [n]-V_L$  one can compute
the conditional probability of the intersection as follows
$$
\mathcal P\left(\bigcap_{v\in J}\mathcal C_{L\ast v}\, |\, \mathcal C_L\right)= \mathcal P(\mathcal C_{L_J}\, | \, \mathcal C_L)= \frac{\mathcal P(\mathcal C_{L_J})}{
\mathcal P( \mathcal C_L)}=
\prod_{v\in J}\left( p_v\prod_{\sigma\in L}p_{\sigma v}\right),
$$
where $L_J$ denotes the simplicial complex $L_J=\cup_{v\in J} L\ast v=L\ast J$ and probabilities above are calculated using the formula (\ref{eq:x}). 

This implies that the family of events $$\{\mathcal C_L -\mathcal C_{L\ast v}\}_{v\in [n]-V_L}$$
is also collectively conditionally independent over $\mathcal C_L$ and gives formula (\ref{eq:cond}). 

We can now use (\ref{eq:7}) and (\ref{eq:cond}) and our assumption (\ref{psigma}) to obtain
\begin{eqnarray}\label{eq:9}
\mathcal P(\mathcal {NQ}_r)&\le& \sum_L \prod_{v\in [n]-V_L} (1-p_v\prod_{\sigma\in L}p_{v\sigma})\nonumber\\
&\le & \binom n {2r+2} \cdot 2^{2^{2r+2}}\cdot (1- p^{3^{r+1}})^{n-2r-2}\nonumber\\
&\le & n^{2r+2}\cdot 2^{2^{2r+2}} \cdot 
(1-p^{3^{r+1}})^{n-2r-2}.
\end{eqnarray}
In the sum above $L$ runs over all non-starlike subcomplexes $L\subset \Delta_n$ with at most $2r+2$ vertices satisfying $|F(L)|\le 3^{r+1}-1$. 
The number $\binom{n}{2r+2}\cdot 2^{2^{2r+2}}$, which appears above, is an upper bound for the number of subcomplexes with at most $2r+2$ vertices.

Because of the assumption (\ref{eq:5}) we have
$$
3^{r+1}\le \log_{(p^{-2})}n \quad \mbox{and}\quad p^{3^{r+1}} \ge p^{\log_{(p^{-2})}n} = n^{-1/2}.
$$
Hence, using (\ref{eq:9}) and the inequality $\ln (1-x)\le -x$ for $x\in (0,1)$, we have
\begin{eqnarray}\label{eq:9c}
\ln \mathcal P(\mathcal {NQ}_r) &\le & (2r+2)\cdot \ln n+ 2^{2r+2}\cdot \ln 2- (n-2r-2)\cdot n^{-1/2}\nonumber \\
&\le &(2r+2)\cdot \ln n+ 2^{2r+2}\cdot \ln 2- 1/2\cdot n^{1/2},
\end{eqnarray}
since $n-2r-2 \ge \frac{n}{2}$ for large $n$. From (\ref{eq:5}) we easily see that $2^{2r+2}\le (\log_{q^{2}}n)^2$ and hence the last term in (\ref{eq:9c}) dominates the other terms.
We conclude that the logarithm $\ln \mathcal P(\mathcal {NQ}_r)$ tends to $-\infty$ and therefore the probability 
$\mathcal P(\mathcal {NQ}_r)$ tends to $0$. 

This completes the proof. 
\end{proof}

Combining Theorem \ref{thm:6} and Theorem \ref{thm8} we obtain:

\begin{corollary}\label{cor:conn}
Under assumption (\ref{psigma}) a random simplicial complex $X\subset \Delta_n$ in the lower model is 
$\lfloor \log_3\log_{q^{2}} n -1\rfloor$-connected, a.a.s.. 
\end{corollary}

\section{Iterated suspensions}

In this section we establish conditions which guarantee that a finite  CW-complex has homotopy type of an iterated suspension. In the following sections we apply these results to random simplicial complexes. 

\begin{lema}\label{lm1b}
Let $X$ be a $(c-1)$-connected CW-complex with $c\ge 3$ and $\dim X = d$. 
If $X$ is homotopy equivalent to a suspension, i.e. $X\simeq \Sigma Y$ for some CW-complex $Y$, 
then $X\simeq \Sigma Y'$ for some $(c-2)$-connected
CW-complex $Y'$ satisfying 
$\dim Y'\le \max\{d-1, 3\}$. 
\end{lema}
\begin{proof}
If $X\simeq \Sigma Y$ then $\tilde H_0(Y)= H_1(X)=0$ and $H_1(Y)= H_2(X)=0$. We see that $Y$ is path-connected and its fundamental group 
$\pi_1(Y)$ is perfect. Applying the plus construction, see \cite{Hatcher}, chapter 4, Proposition 4.40, we obtain a simply connected CW-complex 
$Y^+$ and a map $f: Y\to Y^+$ inducing isomorphism of the homology groups. Hence, $Y^+$ is $(c-2)$-connected. 

By construction, $\dim Y^+\le \max\{\dim Y, 3\}$. 
By Whitehead theorem, the suspension map $\Sigma f: \Sigma Y\to \Sigma (Y^+)$ is a homotopy equivalence and thus we see that $X$ is homotopy equivalent to $\Sigma (Y^+)$. We have $H_i(Y^+)=H_{i+1}(X)=0$ for $i>d-1$. Applying Theorem E of \cite{Wall} we obtain that the complex $Y^+$ is homotopy equivalent to a CW-complex of dimension at most $\max\{d-1, 3\}$. 
\end{proof}

Recall that  Lusternik - Schnirelmann category $\cat(X)$ of a topological space $X$ 
is defined as the smallest integer $k\ge 0$ such that $X$ admits and open cover 
$$X=U_0\cup U_1\cup \dots\cup U_k$$ with the property that each inclusion $U_i\to X$ is null-homotopic, where $i=0, 1, \dots, k$.  
This is {\it the reduced version} of the category which is smaller by one compared to the classical unreduced version. 

The category $\cat(X)$ is a homotopy invariant and for simplicial complexes it is bounded above by the dimension $\cat(X)\le \dim(X)$. Moreover, for an $r$-connected simplicial complex (where $r\ge 0$) one has a stronger upper bound
\begin{eqnarray}\label{dimupper}
\cat(X) \le \frac{\dim(X)}{r+1},
\end{eqnarray}
see \cite{Gro}, \cite{CLOT}.

\begin{teo}\label{thmb}
Let $X$ be a $(c-1)$-connected CW-complex of dimension $d$. If $c\ge 2$ and 
\begin{eqnarray}\label{eq:1}
2c-d\ge k\ge 1,
\end{eqnarray} 
then $X$ has homotopy type of a $k$-iterated suspension,
$\Sigma^k Z$, for some CW-complex $Z$. 
\end{teo}
\begin{proof}  In the case when $d=c$ the complex $X$ is homotopy equivalent to a wedge of $c$-dimensional spheres and then $X\simeq \Sigma^c Z$ for some $Z$, i.e. Theorem \ref{thmb} is true.  

We assume below that $d>c$ and act by induction in $k$ starting with the case $k=1$. 

In the case $k=1$ our assumption (\ref{eq:1}) means that $d=\dim X \le 2c-1$. 
Hence, the Lusternik-Schnirelmann category $\cat(X)$ satisfies 
$$
\cat(X)\le \frac{\dim X}{c} \le 2-\frac{1}{c}<2,
$$
i.e. $\cat(X)\le 1$. Applying Theorem 1.3 from \cite{Ganea} and using the assumption $c\ge 2$ we get the upper bound $\Cat(X)\le 1$ for the strong category. 
Thus, $X$ has homotopy type of a suspension, $X\simeq \Sigma (Y)$, see Proposition 3.16 from \cite{CLOT}. 
This proves our statement for $k=1$. 

Suppose now that $k\ge 2$ and assume that Theorem \ref{thmb} has been proven for all smaller values of $k$. If $X$ is $(c-1)$-connected, and $d=\dim X > c$ satisfies
$2c-d\ge k\ge 2$, then $2c \ge 2+d\ge 3+c$, i.e. $c\ge 3$. Besides, $d>c\ge 3$ implies $d>3$. 
Repeating the arguments above we see that $X$ is homotopy equivalent to a suspension, and 
applying Lemma \ref{lm1b}, we find that $X\simeq \Sigma Y',$ where $Y'$ is a $(c'-1)$-connected CW-complex of dimension $d'=d-1$ and $c'=c-1$. 
Since 
$$
2c'-d'=2c-d-1\ge k-1,
$$
we may apply induction and get $Y'\simeq \Sigma^{k-1}Z$ for some $Z$. Thus, $X\simeq \Sigma^kZ$ as claimed. 

This completes the proof. 
\end{proof}

%
%
%
%

\section{Random simplicial complexes in the medial regime as iterated suspensions}


{\it The medial regime} \cite{FM} is a special case of the probabilistic model described in \S \ref{sec:31} in which the probability parameters $p_\sigma$ satisfy the inequalities
\begin{eqnarray}\label{eq:11}
0<p\, \le\, p_\sigma\, \le\,  P<1,
\end{eqnarray}
where $0<p\le P<1$ are constants independent of $n$. 
In other words, as in \S \ref{sec:3}, we are dealing with the probability measure $\mathcal P$ on the set of all simplicial subcomplexes $X\subset \Delta_n$ in the lower multi-parameter model which can be uniquely characterised by the property (\ref{eq:x}) in which $p_\sigma\in [0,1]$ are probability parameters associated to each simplex $\sigma\subset \Delta_n$ which satisfy (\ref{eq:11}), i.e. they are not allowed to approach $0$ and $1$. 
We shall use the following result established in Theorem 3.1 of \cite{FM}:
\begin{teo}\label{thm10}
The dimension of a random simplicial complex $X\subset \Delta_n$ in the medial regime satisfies 
\begin{eqnarray*}\label{dimupper1}
\dim(X)  \, < \,  \log_2\ln n +\log_2\log_2\ln n- \log_2 A, \quad \mbox{where}\quad P=e^{-A},
\end{eqnarray*}
asymptotically almost surely.
\end{teo}

We may apply Corollary \ref{cor:conn} 
and find the level of connectivity of a random simplicial complex in the medial regime, a.a.s. 
Indeed, writing $p=e^{-a}$, where $a>A>0$, we find that 
$$
\log_3\log_{q^2}n= \log_3 2\cdot \log_2\ln n - \log_3(2a), 
$$
where $\log_3 2\sim 0.63092975$. Thus, we see that a random simplicial complex $X\subset \Delta_n$ in the medial regime is $(c-1)$-connected, a.a.s.,
with 
\begin{eqnarray*}\label{c}
c=\lfloor  
  \log_3 2\cdot \log_2\ln n - \log_3(2a)  
\rfloor.
\end{eqnarray*}
\begin{teo}\label{thm14}
A random simplicial complex $X\subset \Delta_n$ in the medial regime has homotopy type of a $k$-fold iterated suspension, where $k=\lfloor 0.26 \cdot \log_2\ln n\rfloor$, a.a.s.
\end{teo}
\begin{proof} We use Theorem \ref{thm8} combined with Theorem \ref{thmb} and Theorem \ref{thm10}. 
Setting $$d=\lfloor \log_2\ln n +\log_2\log_2\ln n- \log_2 A\rfloor$$
we find that for large $n$, 
\begin{eqnarray*}
2c-d &\ge & \log_3\left(\frac{4}{3}\right) \cdot \log_2\ln n -\log_2\log_2\ln n +\log_2 A -2\log_3(2a)-2\\
&\ge& 0.26 \cdot \log_2\ln n.
\end{eqnarray*}
Our claim now follows from Theorem \ref{thmb}. 
\end{proof}
Theorem \ref{thm14} clearly implies:

\begin{teo}\label{thm:thm1}
The Lusternik-Schnirelmann category $\cat(X)$ of a random simplicial complex $X\subset \Delta_n$ in the medial regime satisfies
$
\cat(X)\le 1,
$
a.a.s. 
\end{teo}

%
%
%
%
%
%

{\it The topological complexity} $\tc(X)$ of a topological space $X$ is defined as the minimal integer $k\ge 0$ such that $X\times X$ admits an open cover 
$X\times X=U_0\cup U_1\cup \dots \cup U_k$ with the property that each inclusion $$\mu_i: U_i\to X\times X, \quad i=0, 1, \dots, k$$ satisfies 
$${\pr}_1\circ \mu_{i}\simeq {\pr}_2\circ \mu_i, \quad i=0, 1, \dots, k$$ where
${\pr}_1, {\pr}_2: X\times X\to X$ are the projections on the first and the second factor correspondingly and the sign $\simeq$ stands for homotopy. 
We refer to \cite{Far0} where this notion was introduced;
see also \cite{Far2},  Lemma 4.21 which explains that the above definition is equivalent to the one given in \cite{Far0}. 
We emphasise, that in this paper we use {\it the reduced version} of the topological complexity which is smaller by 1 compared to the definition of \cite{Far0}. 

\begin{teo}\label{thm:tcle2}
The topological complexity  $\tc(X)$ of a random simplicial complex $X\subset \Delta_n$ in the medial regime satisfies
$
\tc(X)\le 2,
$
a.a.s.
\end{teo}
\begin{proof}
This follows from the general inequality 
$
\tc(X)\le 2\cdot \cat(X),
$
see \cite{Far0}, Theorem 5, and from Theorem \ref{thm:thm1}. 
\end{proof}

\marginpar{revise}
We shall finish this section with remarks related to the opposite inequality $\tc(X)\ge 2$.

The main result of \cite{GLO} states that a space of topological complexity one has homotopy type of an odd-dimensional sphere. 
In this paper we shall use the following Lemma which is weaker but has a simple independent proof.  

\begin{lema}\label{lm:tc2}
For a finite CW-complex $X$ either of the following conditions (a), (b), (c) implies $\tc(X)\ge 2$:
\begin{enumerate}
\item[{(a)}] An even-dimensional Betti number satisfies $b_i(X)\ge 1$, where $i\ge 2$ is even;
\item[{(b)}] An odd-dimensional Betti number satisfies $b_i(X)\ge 2$, where $i\ge 1$ is odd;
\item[{(c)}]  Two distinct odd-dimensional Betti numbers satisfy $b_i(X)\ge 1$ and $b_j(X)\ge 1$, where $i, j\ge 1$ are odd, $i\not=j$. 
\end{enumerate}
\end{lema}
\begin{proof}
We shall use the cohomological lower bound for the topological complexity stated as Corollary 4.40 in \cite{Far2}. We need to remember that in \cite{Far2}  the author used unreduced version of the topological complexity. 

Assuming (a), there is a nonzero even-dimensional rational cohomology class $u\in H^i(X;\Bbb Q)$, where $i\ge 2$. Then the class 
$\overline u= u\times 1-1\times u\in H^i(X\times X;\Bbb Q)$ is a zero-divisor and its square equals  
$\overline u^2=u^2\times 1+1\times u^2 -2u\times u.$ 
We refer to standard material on cross-products to \cite{S}, chapter 5, \S 6. 
To show that $\overline u^2\not=0$, consider a homology class $z\in H_i(X;\Bbb Q)$ such that 
$\langle u, z\rangle \not=0$. Then one has $\langle \overline u^2, z\times z\rangle = -2 \langle u, z\rangle^2\not=0$ implying $\overline u^2\not=0$. 
Corollary 4.40 from \cite{Far2} now implies $\tc(X)\ge 2$. 

Assuming (b), there exist two linear independent cohomology classes $u, v\in H^i(X;\Bbb Q)$. The product of the corresponding zero-divisors equals
$$
\overline u\cdot \overline v = (uv)\times 1+1\times (uv) -u\times v+v\times u \in H^{2i}(X\times X;\Bbb Q).
$$
Let $z_u, z_v\in H_i(X;\Bbb Q)$ be homology classes dual to $u, v$, i.e. $\langle u, z_u\rangle =1=\langle v, z_v\rangle$ and $\langle u, z_v\rangle=0=\langle v, z_u\rangle$. Then 
$\langle \overline u\cdot \overline v, z_u\times z_v\rangle =- \langle u\times v, z_u\times z_v\rangle = - 1$ and hence $\overline u\cdot \overline v\not=0$. 
The result follows by applying Corollary 4.40 from \cite{Far2}. 

The proof under assumption (c) is similar. 
\end{proof}

\begin{lema}\label{lm:tc22}
Suppose that the face numbers $f_j(X)$ of  a finite simplicial complex $X$ satisfy 
\begin{eqnarray}\label{eq:15c}
f_{\ell}(X)-f_{\ell-1}(X)-f_{\ell+1}(X)\ge 2,\quad \mbox{for some}\quad \ell.
\end{eqnarray}   
Then $\tc(X)\ge 2$. 
\end{lema}
\begin{proof} Consider the simplicial chain complex of $X$ over $\mathbb Q$
$$
\dots\to C_{j+1}\stackrel d\to C_j\stackrel d\to C_{j-1}\to \dots, \quad \dim C_j=f_j(X).
$$
Denoting by $\beta_j$ the dimension of the image of $d:C_{j+1}\to C_j$ we have 
$$
f_j(X) = \dim C_j=b_j(X)+\beta_j +\beta_{j-1}.
$$
Substituting this into  (\ref{eq:15c}) gives
$$
b_\ell(X) -b_{\ell-1}(X)-b_{\ell+1}(X)-\beta_{\ell-2}-\beta_{\ell+1}\ge 2,
$$
which implies  $b_\ell(X)\ge 2.$ Now our statement follows from Lemma \ref{lm:tc2}. 
\end{proof}

%
%

\section{HMR random complexes and their critical dimensions}\label{sec:hmr}
\subsection{} In this and in the following sections we consider a special class of random simplicial complexes (as described in \S\ref{sec:31}), where each probability parameter $p_\sigma$ equals a fixed number 
$p\in (0,1)$, i.e. 
we assume that 
\begin{eqnarray*}
p_\sigma=p\quad \mbox{ for all}\quad  \sigma\in \Delta_n.
\end{eqnarray*} 
We refer to this model as {\it homogeneous medial regime random simplicial complex}, or HMR random complex, for short.
Our main result states: 

\begin{teo}\label{thm:tc2}
For a subset $\N'\subset \N$ of density 1, the probability that an HMR random simplicial complex $X\subset \Delta_n$ satisfies $\tc(X)=2$ tends to one as $n\to \infty$, $n\in \N'$. 
\end{teo} 

Recall that a subset $\N'\subset \N$ has density 1 if $$n^{-1}\cdot \sharp(\N'\cap [0,n])\to 1, \quad \mbox{ as}\quad  n\to \infty.$$ 
Here the symbol $\sharp(\N'\cap [0,n])$ stands for the cardinality of the set $\N'\cap [0,n]$. 

\subsection{} Note that under assumptions of Theorem \ref{thm:tc2} one has $\cat(X)=1$ with probability tending to one as $n\to \infty$, $n\in \N'$. Indeed, due to Theorem 
\ref{thm:thm1} we know that $\cat(X)\le 1$ and the possibility $\cat(X)=0$ would imply the contractibility of  $X$ and hence $\TC(X)=0$, 
 contradicting Theorem \ref{thm:tc2}. 
 
 \subsection{} Note that spaces of topological complexity $\tc(X)$ equal $0$ or $1$ are quite simple: $\tc(X)=0$ characterises contractible spaces and $\tc(X)=1$ happens iff $X$ has homotopy type of a sphere of odd dimension, see \cite{GLO}. 
 Homotopy types of spaces with $\tc(X)=2$ form a large variety, as Theorem \ref{thm:tc2} may indicate. For example, any $r$-connected 
 finite CW-complex $X$ of dimension $\dim X\le 2r+1$ having total Betti number greater than 1 satisfies $\tc(X)=2$. The upper bound $\tc(X)\le 2$ follows from inequality (\ref{dimupper}) and the lower bound follows from Lemma \ref{lm:tc2}. 

The proof of Theorem \ref{thm:tc2} will occupy the rest of the paper, it will be completed in \S \ref{sec:9}.


\subsection{} We shall 
analyse the face numbers $f_\ell(X)$, i.e. the numbers of $\ell$-dimensional simplexes in $X$. The face numbers $f_\ell$ are random variables on the probability space of random simplicial complexes, and our first goal will be to analyse their mathematical expectations $\E(f_\ell)$.

Probability that a simplex $\sigma\in \Delta_n$ of dimension $\ell$ is included in a HMR random simplicial complex $X\subset \Delta_n$ equals 
$
p^{2^{\ell+1}-1},
$
compare (\ref{eq:x}). Thus, the expectation of $f_\ell$ is given by 
\begin{eqnarray}\label{eq:expect}
\mathbb E(f_\ell) =\binom n {\ell+1} \cdot p^{2^{\ell+1}-1}, \quad \mbox{where}\quad \ell=0, 1, \dots, n-1.
\end{eqnarray}
We shall denote $q=p^{-1}.$ Note that $\E(f_\ell)$ is a function of two parameters, $\ell$ and $n$. 

For an integer $\ell\ge 0$, let $J_\ell$ denote the set $J_\ell\subset \N$,
\begin{eqnarray}\label{eq:16}
J_\ell =\{n\in \N; (\ell+1) q^{2^{\ell}}+\ell \, < n\,  <  (\ell+2) q^{2^{\ell+1}}+\ell+1\}.
\end{eqnarray}
The sets $J_\ell$ are pairwise disjoint, i.e. $J_\ell\cap J_{\ell'}=\emptyset$ for $\ell\not=\ell'$. 


\begin{lema}  \label{lm:14a}
For a fixed integer $n\in J_\ell$, the sequence $\E(f_0), \E(f_1), \dots, \E(f_{n-1})$ is unimodal. More precisely, if $n\in J_\ell$ then one has 
$$
\E(f_0)<\E(f_1)< \dots<\E(f_{\ell-1})<\E(f_{\ell})>\E(f_{\ell+1})>\dots >\E(f_{n-1}),
$$
i.e. the sequence $\E(f_i)$ attains its maximum for $i=\ell$ and it is monotone increasing for $i\le \ell$ and monotone decreasing for $i\ge \ell$. 
\end{lema} 
\begin{proof} The ratio
\begin{eqnarray}\label{eq:ratio1a}
 \frac{\mathbb E(f_{i+1})}{\mathbb E(f_{i})}= \frac{n-i-1}{i+2}\cdot p^{2^{i+1}}
\end{eqnarray}
is obviously a decreasing function of $i$. From (\ref{eq:ratio1a}) we see that $\E(f_{i+1})>\E(f_i)$ iff  
$$n>(i+2)q^{2^{i+1}}+i+1,$$ i.e. iff $i\le \ell-1$, as seen from 
(\ref{eq:16}). Similarly, one has $\E(f_{i+1})<\E(f_i)$ iff $$n<(i+2)q^{2^{i+1}}+i+1,$$ i.e. iff $i\ge \ell$. 
\end{proof}
We see that, as long as $n\in J_\ell$, the expectation of the face number $f_{\ell}$ dominates the expectations of all other face numbers. 
Since the intervals $J_\ell$ are pairwise disjoint, any fixed $n\in \N$ lies in at most one interval $J_\ell$. 

\begin{definition}
If $n\in J_\ell$ we shall say that $\ell$ is {\it the critical dimension} of an HMR random simplicial complex $X\subset \Delta_n$. 
\end{definition}

\begin{lema}\label{lm:22}
For $n\in J_\ell$ with $\ell$ large enough the critical dimension $\ell$ can be represented in the form 
\begin{eqnarray}\label{eq:19b}
\ell=\log_2\log_q n  -\epsilon_n, \quad \mbox{where}\quad q=p^{-1}.\end{eqnarray} 
The function $n\mapsto \epsilon_n$, where $n\in J_\ell$, is monotone increasing and one has 
$\epsilon_n\in (\alpha_\ell^-,\alpha_\ell^+)$, where 
\begin{eqnarray*}
\alpha_\ell^-=\log_2\left[1+\frac{\log_q\ell}{2^\ell}\right]\quad \mbox{and}\quad 
\alpha_\ell^+ = 1+ \frac{\log_q(2\ell)}{\ln 2 \cdot 2^{\ell+1}}.
\end{eqnarray*}
In particular, $\alpha^-_\ell>0$ and $\alpha^-_\ell\to 0$; besides, 
$\alpha^+_\ell>1$ and $\alpha^+_\ell\to 1$ when $\ell\to \infty$. 
\end{lema}
\begin{proof} Writing $\epsilon_n=\log_2\log_qn-\ell$ we obviously see that $\epsilon_n$ is monotone increasing. For $n\in J_\ell$ one has $n>\ell q^{2^{\ell}}$
and hence $\log_qn >2^\ell+ \log_q\ell= 2^\ell\cdot \left[1+\frac{\log_q\ell}{2^\ell}\right]$. Applying $\log_2$ we get 
$\log_2\log_qn>\ell +\log_2\left[1+\frac{\log_q\ell}{2^\ell}\right]$. This proves that $\epsilon_n>\alpha_\ell^-$ for $n\in J_\ell$. 

Similarly, for $n\in J_\ell$ with $\ell$ large,
we have $n<(2\ell)\cdot q^{2^{\ell+1}}$. Applying $\log_q$ we get 
$$
\log_q n<2^{\ell+1} +\log_q(2\ell)= 2^{\ell+1}\cdot \left[1+ \frac{\log_q(2\ell)}{2^{\ell+1}}\right]
$$
and hence
$$
\log_2\log_q n < \ell+1 + \log_2\left[1+ \frac{\log_q(2\ell)}{2^{\ell+1}}\right] < \ell+1 +\frac{\log_q(2\ell)}{\ln 2\cdot 2^{\ell+1}}.
$$
This proves that $\epsilon_n<\alpha_\ell^+$ for $n\in J_\ell$ with $\ell\ge 2$. 
\end{proof}
%
%
%
%


Finally we shall also mention an estimate for $\E(f_\ell)$ for $n\in J_\ell$. 
\begin{lema}
For $n\in J_\ell$ one has 
$$\E(f_\ell)\ge  \frac{q \cdot n^{\ell-1}}{(\ell+1)^{\ell-1}}.$$
\end{lema}
\begin{proof}
From (\ref{eq:16}) we see that for $n\in J_\ell$ one has $q^{2^{\ell}}<\frac{n-\ell}{\ell+1}$ and hence $p^{2^{\ell}}>\frac{\ell+1}{n-\ell}>\frac{\ell+1}{n}$. Hence using the well-known inequality $\binom a b \ge \frac{a^b}{b^b}$ we obtain
\begin{eqnarray*}\label{ineq:29}
\E(f_\ell) =\binom n {\ell+1} \cdot p^{2^{\ell+1}-1} > q \cdot \binom n {\ell+1}\left(\frac{\ell+1}{n}\right)^2>   \frac{q \cdot n^{\ell-1}}{(\ell+1)^{\ell-1}}. 
\end{eqnarray*}
\end{proof}



\section{The intervals of strong domination}

In this section we introduce  a smaller interval 
$I_{\ell, c, \kappa}\subset J_\ell$, where {\it \lq\lq strong domination\rq\rq}\   takes place.
\subsection{} For real numbers 
$$0<c<1/2\quad\mbox{ and}\quad  1/2 <\kappa<1,$$ let $I_{\ell,c, \kappa}\subset \Bbb N$ denote 
the set of integers $n\in \N$ satisfying the inequalities
\begin{eqnarray}\label{in:26a}
 \frac{2\kappa(\ell+1)}{1-c} \cdot q^{2^{\ell}}\, \le \, n-\ell \, \le
 (1-c)(\ell+1) \cdot q^{2^{\ell+1}}\cdot\left[1-\frac{2\kappa p^{2^{\ell}}}{(1-c)^2}\right].
 \end{eqnarray}
 Obviously, $I_{\ell, c, \kappa}\subset J_\ell$. For the cardinality $|I_{\ell, c, \kappa}|$ one has
 \begin{eqnarray}\label{eq:ilc}
|I_{\ell, c, \kappa}| \ge  (1-c)(\ell+1)q^{2^{\ell+1}}\cdot \left[1-\frac{4\kappa p^{2{^\ell}}}{(1-c)^2}\right]-1 
 \end{eqnarray} 
\begin{lema} For any $n\in I_{\ell, c, \kappa}$ 
one has
\begin{eqnarray}\label{ineq:19a}
\Bbb E(f_{\ell})-\Bbb E(f_{\ell-1}) - \Bbb E(f_{\ell+1})\ge c \cdot\Bbb E(f_{\ell}),
\end{eqnarray}
assuming that \begin{eqnarray}\label{in:new1}
1+\frac{4\cdot p{^{2^\ell}}}{(1-c)^2}<2\kappa<2.
\end{eqnarray} 
\end{lema}
\begin{proof}
Rewriting (\ref{ineq:19a}) in the form
$$
\frac{\E(f_{\ell+1})}{\E(f_{\ell})}\cdot \frac{\E(f_{\ell})}{\E(f_{\ell-1})}-(1-c)\cdot\frac{\E(f_{\ell})}{\E(f_{\ell-1})}+1\le 0
$$
and using (\ref{eq:ratio1a}) we obtain
\begin{eqnarray*}
 \frac{n-\ell-1}{\ell+2}\cdot p^{2^{\ell+1}}\cdot \frac{n-\ell}{\ell+1}\cdot p^{2^{\ell}}- (1-c)\cdot \frac{n-\ell}{\ell+1}\cdot p^{2^{\ell}}+1\le 0.
\end{eqnarray*}
The above inequality is a consequence of the following:
\begin{eqnarray*}
 \left(\frac{n-\ell}{\ell+1}\right)^2\cdot p^{3\cdot 2^{\ell}}-(1-c)\cdot \left( \frac{n-\ell}{\ell+1}\right) \cdot p^{2^{\ell}}+1\le 0.
\end{eqnarray*}
The latter inequality can be rewritten as 
\begin{eqnarray}\label{eq:quadratica}
\left(\frac{n-\ell}{\ell+1}\right)^2 -(1-c)\cdot \left(\frac{n-\ell}{\ell+1}\right)\cdot q^{2\cdot 2^{\ell}}+  q^{3\cdot 2^\ell}\le 0. 
\end{eqnarray}
The roots of the quadratic equation $$x^2-(1-c)\cdot x\cdot q^{2\cdot 2^{\ell}} +q^{3\cdot 2^\ell}=0$$ are 
$$
x=\frac{n-\ell}{\ell+1} = \frac{(1-c)}{2}\cdot q^{2\cdot {2^\ell}}\cdot\left(1\pm \sqrt{1-\alpha}\right), \quad \mbox{where}\quad \alpha=\frac{4\cdot p^{2^\ell}}{(1-c)^2}.
$$
Inequality (\ref{in:new1}) gives 
\begin{eqnarray*}\label{eq:est}
\alpha <2\kappa-1 \quad \mbox{and}\quad \kappa <1,
\end{eqnarray*}
 and hence we can 
estimate the square root as follows
$$
1-\kappa\cdot \alpha \le \sqrt{1-\alpha}.
$$
Taking into account that the quadratic function (\ref{eq:quadratica}) is negative between the roots, we arrive at the following conclusion: for $n$ satisfying the inequality
\begin{eqnarray}\label{in:24}
\frac{2\kappa q^{2^\ell}}{1-c} \, \le \frac{n-\ell}{\ell+1}\, \le (1-c)q^{ 2^{\ell+1}}-  \frac{2\kappa q^{2^\ell}}{1-c}
\end{eqnarray}
the strong domination (\ref{ineq:19a}) holds. Clearly, the set of solutions to (\ref{in:24}) coincides with $I_{\ell, c, \kappa}$, as defined in (\ref{in:26a}). 
\end{proof}
\subsection{} For $\ell=0, 1, 2, \dots,$ let $\ell \mapsto c(\ell)$ and $\ell\mapsto \kappa(\ell)$ be two functions satisfying 
$$0<c(\ell)<1/2\quad\mbox{and}\quad  
1/2<\kappa(\ell)<1$$ 
and 
$$c(\ell)\to 0,\quad\mbox{}\quad \kappa(\ell)\to 1/2,\quad \mbox{when}\quad \ell\to \infty.$$ 
Additionally, we shall assume that for all large $\ell$ one has
\begin{eqnarray}\label{eq:30a}
\kappa(\ell)  >1/2 + 8p^{2^{\ell}},
\end{eqnarray}
which guarantees (\ref{in:new1}). 
Denote
\begin{eqnarray*}\label{eq:nprime}
\N'=\bigsqcup_{\ell\ge 0} I_{\ell, c(\ell), \kappa(\ell)}.
\end{eqnarray*}
\begin{lema}
The density of the set $\N'\subset \N$ equals $1$. 
\end{lema} 
\begin{proof} Consider the density function
$$
\psi(m) =\frac{ \sharp([0, m]\cap \N')}{m}, \quad m\in \N. 
$$
Since the set $\N'$ is a union of intervals $I_{\ell, c(\ell), \kappa(\ell)}$, it is easy to see that 
$$\psi(m)\ge \psi(m-1) \quad \mbox{if}\quad  m\in \N'$$
 and 
 $$\psi(m)\le  \psi(m-1) \quad \mbox{if}\quad  m\notin \N'.$$
 Thus, 
 $$\liminf_{m\to \infty} \psi(m) = \liminf_{\ell\to \infty} m_\ell,$$ where 
 $$m_\ell= \inf I_{\ell, c(\ell), \kappa(\ell)}-1. $$
By (\ref{in:26a}),
we have
$$
m_\ell\le  \frac{2\kappa(\ell) (\ell+1)}{1-c(\ell)}\cdot q^{2^{\ell}}+\ell -1.
$$ 
Taking into account (\ref{eq:ilc}), we obtain, where for the sake of compactness we write $c'=c(\ell-1)$ and $\kappa'=\kappa(\ell-1)$:
\begin{eqnarray*}
\sharp ([0, m_\ell]\cap \N') &\ge& |I_{\ell-1, c',\kappa'}|
\ge 
(1-c')\cdot \ell \cdot q^{2^{\ell}}\cdot \left[1-\frac{4\kappa'\cdot p^{2^{\ell-1}}}{(1-c')^2}\right]-1,\end{eqnarray*}
and hence
\begin{eqnarray*}
\psi(m_\ell) \ge \frac{(1-c')\cdot \ell \cdot q^{2^{\ell}}\cdot \left[1-\frac{4\kappa' p^{2^{\ell-1}}}{(1-c')^2}\right]-1}{\frac{2\kappa(\ell) \cdot (\ell+1)}{1-c(\ell)}\cdot q^{2^{\ell}}+\ell -1}.
\end{eqnarray*}
Since $c\to 0$ and $\kappa\to 1/2$, we see that the RHS of this inequality converges to 1. 
Since obviously $\psi(m)\le 1$, we conclude that $\lim_{m\to \infty} \psi(m)=1$, as claimed. 
\end{proof}

\section{Concentration around the expectation}

In this section we analyse probability that the face numbers $f_j(X)$ of an HMR random simplicial complex 
(see \S \ref{sec:hmr})
deviate significantly from its expectation $\E(f_j)$. We shall be mainly interested in the case when the dimension $j$ equals either $\ell, \ell+1, \mbox{or}\, \ell-1$, where $\ell$ is the critical dimension. 
Our main tool is the Chebyshev inequality 
\begin{eqnarray}\label{in:32f}
\mathcal P(|f_j(X) - \mathbb E(f_j)|\ge \mu \cdot \mathbb E(f_j)) \le \frac{1}{\mu^2}\cdot \frac{{\rm {Var}}(f_j)}{\mathbb E(f_j)^2}
\end{eqnarray}
which estimates the probability of the inequality $|f_j(X) - \mathbb E(f_j)|\ge \mu \cdot \mathbb E(f_j)$. Recall that the variance  
${\rm {Var}}(f_j)$ is defined as 
${\rm {Var}}(f_j)=\E(\left [f_j-\E(f_j)\right ]^2)=\E(f_j^2)-\E(f_j)^2.$

\begin{lema}\label{lm:26}
Consider an HMR random simplicial complex $X$ with the number of vertexes $n$ large enough, $n\to \infty$. Then for $j\le \ell+1$, where $\ell$ is the critical dimension, one has 
$$
\frac{{{\rm {Var}}(f_j)}}{\E(f_j)^2}\le \frac{3 (\ell+2)^{2(\ell+2)}}{n}.
$$
\end{lema}

\begin{proof} 
%
 
Following \cite{FM}, page 13, one has
$$
\mathbb E(f^2_j)= \sum_{i=0}^{j+1}\binom n {j+1}\cdot \binom {j+1}{i} \cdot \binom {n-j-1}{j+1-i}\cdot p^{2\cdot 2^{j+1}-2^i-1}
$$
and 
\begin{eqnarray*}\label{eq:30}
\frac{\mathbb E(f^2_j)}{\mathbb E(f_j)^2}= \sum_{i=0}^{j+1}\frac{\binom {j+1}{i} \cdot \binom {n-j-1}{j+1-i}}{\binom n {j+1}}\cdot p^{-2^i+1}.
\end{eqnarray*}
Denote 
$$
r_i= \frac{\binom {j+1}{i} \cdot \binom {n-j-1}{j+1-i}}{\binom n {j+1}}\cdot p^{-2^i+1}\le \frac{\binom {j+1}{i} \cdot \binom {n-j-1}{j+1-i}}{\binom n {j+1}}\cdot q^{2^i}
$$
where $q=p^{-1}$ and $ i=0, 1, \dots, j+1.$
For the term $r_0$ we obviously have
\begin{eqnarray*}\label{eq:r0}
r_0= \frac{\binom {n-j-1}{j+1}}{{\binom{n} {j+1}}}<1.
\end{eqnarray*}

The term $r_1$ can be estimated as follows
\begin{eqnarray*}\label{eq:r1}
r_1&=& \frac{(j+1)\cdot \binom {n-j -1}{j}}{\binom n {j+1} } \cdot q = \frac{(j+1)\cdot \binom {n-j -1}{j}}{\frac{n}{j+1} \cdot \binom {n-1} j  }\cdot q \nonumber \\
&\le & \frac{(j+1)^{2}}{n}\cdot q\le \frac{q\cdot (\ell+2)^{2}}{n}.
\end{eqnarray*}
Next we examine the terms $r_i$ with $i=2, 3, \dots, j+1$. 
We use below the well-known inequality
$
\frac{a^b}{b^b}\le \binom a b\le a^b
$
and the fact that the function $x\mapsto \frac{2^x}{x}$ is increasing for $x\ge 2$. We have
\begin{eqnarray*}
r_i&\le& \frac{(j+1)^{j+1+i}}{n^i}\cdot q^{2^i}
\le (\ell+2)^{2(\ell+2)}\cdot \left[\frac{q^{\frac{2^i}{i}}}{n}\right]^i \\
&\le & (\ell+2)^{2(\ell+2)}\cdot \left[\frac{q^{\frac{2^{\ell+2}}{\ell+2}}}{n}\right]^i 
\le  
(\ell+2)^{2(\ell+2)}\cdot n^{-i/2}.
\end{eqnarray*}
We used above our assumption $i\le j+1\le \ell+2$, Lemma \ref{lm:22}, and the following inequality 
\begin{eqnarray*}\label{eq:beta-1}
\frac{2^{\ell+2}}{\ell+2}\le  \frac{4\cdot\log_q n}{\ell+2} \le 
\frac{1}{2}\cdot \log_qn
\end{eqnarray*}
which is valid for $\ell\ge 6$. 
Thus, we obtain, assuming that $n\ge 4$, 
\begin{eqnarray*}
r_2+r_3+\dots+r_{j+1} &\le & (\ell+2)^{2(\ell+2)}\cdot \left[n^{-1}+n^{-3/2}+\dots+n^{-(j+1)/2}\right]\nonumber \\
&\le& \frac{(\ell+2)^{2(\ell+2)}}{n}\cdot \frac{1}{1-n^{-1/2}}\le \frac{2\cdot (\ell+2)^{2(\ell+2)}}{n}.\label{eq:ri}
\end{eqnarray*}
Hence, 
$$
\frac{{{\rm {Var}}(f_j)}}{\mathbb E(f_j)^2}=\frac{\E(f_j^2)}{\E(f_j)^2} -1 = r_0+r_1+\dots+r_{j+1}-1 \le \frac{3\cdot (\ell+2)^{2(\ell+2)}}{n}.
$$
This completes the proof. 
\end{proof}
\begin{corollary}\label{cor:27}
Suppose that $\mu=\mu(n)>0$ is such that 
\begin{eqnarray}\label{eq:38}
\frac{(\ell+2)^{2(\ell+2)}}{n \mu^2} \to 0\quad \mbox{as}\quad  n\to \infty.
 \end{eqnarray} 
 Then, with probability tending to 1 as 
$n\to \infty$, $n\in \N'$, 
for an HMR random simplicial complex $X\subset \Delta_n$ the following 3 inequalities hold:
$$
f_\ell(X)>(1-\mu)\cdot \E(f_\ell), \quad f_{\ell-1}(X)<(1+\mu)\cdot \E(f_{\ell-1}), \quad  f_{\ell+1}(X)<(1+\mu)\cdot \E(f_{\ell+1}). 
$$
\end{corollary}
\begin{proof}
This follows from the inequality (\ref{in:32f}) and Lemma \ref{lm:26}. 
\end{proof}

\section{Proof of Theorem \ref{thm:tc2}}\label{sec:9}

\begin{proof}
Set $\mu=n^{-\nu}$, where $0<\nu<1/2$; for instance, one can take $\nu=1/4$. Then (\ref{eq:38}) is safisfied. 
We can also set 
\begin{eqnarray}\label{eq:40d}
c=c(\ell)= 6\cdot \max\{n^{-\nu};\  n\in J_\ell\}.
\end{eqnarray} 
Besides, set $$\kappa= 1/2 + 9p^{2^{\ell}};$$ compare (\ref{eq:30a}). 

If $n\in I_{\ell, c, \kappa}$ then by Corollary \ref{cor:27}, with probability tending to $1$, one has
\begin{eqnarray*}
f_\ell(X)-f_{\ell-1}(X)-f_{\ell+1}(X) &>& (1-\mu)\cdot \E(f_\ell) -(1+\mu)\cdot \E(f_{\ell-1}) -(1+\mu)\cdot \E(f_{\ell+1})\\
&=& \E(f_\ell)-\E(f_{\ell-1})-\E(f_{\ell+1}) - \mu\cdot(\E(f_\ell) +\E(f_{\ell-1})+\E(f_{\ell+1}))\\
&\ge & c\cdot \E(f_\ell) - \mu\cdot(\E(f_\ell) +\E(f_{\ell-1})+\E(f_{\ell+1}))\\
&\ge& (c/3 -\mu)\cdot (\E(f_\ell) +\E(f_{\ell-1})+\E(f_{\ell+1}))\\
&\ge & \mu \cdot (\E(f_\ell) +\E(f_{\ell-1})+\E(f_{\ell+1}))\\
&\ge & n^{-\nu}\cdot \frac{qn^{\ell-1}}{(\ell+1)^{\ell-1}}\to \infty. 
\end{eqnarray*}
We used the inequality (\ref{ineq:19a}) and the observation that for $n\in I_{\ell,c, \kappa}$ one has $c(\ell) \ge 6\mu$, as one can see from (\ref{eq:40d}). 
Lemma \ref{lm:22} is used here twice, firstly to ensure that the assumptions of Corollary \ref{cor:27} are satisfied and secondly to get 
the limit on the last line above. 
Applying  Lemma \ref{lm:tc22}, we see that probability that a random complex $X\subset \Delta_n$ satisfies $\tc(X)\ge 2$ tends to 1 as $n\to \infty$, $n\in \N'$. 
\end{proof}

\section{Open questions}

Comparing Theorem \ref{thm:tcle2} which gives an upper bound $\tc(X)\le 2$ for the medial regime random simplicial complexes with 
Theorem \ref{thm:tc2} which (1) applies to a smaller class of HMR random simplicial complexes and (2) it restricts the set of values of the number of vertexes $n$ to a subset $\N'\subset \N$, we come to the following questions: 

1. Can Theorem \ref{thm:tc2} be extended to include all random simplicial complexes in the medial regime and not only HMR random simplicial complexes? 

2. Can Theorem \ref{thm:tc2} be strengthened by not limiting $n$ to lie in a subset $\N'\subset \N$ of density 1? In other words, can one take $\N'=\N$? 

We hope that these questions will be resolved in the future by using appropriate new techniques.  

The reason why we need to restrict the number of vertexes $n$ to the subset $\N'\subset \N$ is that for $n\notin \N'$ two subsequent face numbers $f_\ell$ and $f_{\ell+1}$ of a random complex are of similar magnitude and hence no significant homology in these two dimensions can be easily predicted. On the other hand, it would be extremely interesting and surprising
if for an infinite sequence of integers $n_i\to \infty$ the probability that an HMR random simplicial complex $X\subset \Delta_{n_i}$ satisfies 
$\tc(X)=2$ does not converge to $1$.


\begin{thebibliography}{99}


\bibitem{Bar} J.A. Barmak, \textit{Connectivity of Ample, Conic, and Random Simplicial Complexes}. Int. Math. Res. Not. IMRN 2023, no. 8, pp. 6579 -- 6597.

\bibitem{Bar2} J. A. Barmak, \textit{On the connectivity of conic complexes},  J. Appl. Comput. Topol. 8 (2024), no. 6, 1571–1574.

\bibitem{Bat} F. Battiston, G. Cencetti, I. Iacopini, V. Latora, M. Lucas, A. Patania, J.-G. Young, G. Petri, \textit{Networks beyond pairwise interactions: structure and dynamics}. Phys. Rep. 874 (2020), 1–92

\bibitem{Bjo} A. Bj\"orner, Topological methods, Handbook of Combinatorics (R. Graham, M. Grötschel, and L. Lovász, eds.), North-Holland, Amsterdam, 1994, pp. 1819-1872.

\bibitem{BK} O. Bobrowski, D. Krioukov, 
\textit{Random simplicial complexes: models and phenomena,} Higher-order systems, pp. 59 -- 96,
Underst. Complex Syst., Springer, Cham, 2022. 

\bibitem{Bol} B. Bollobas, \textit{Random Graphs}, Cambridge University Press, 2001.

\bibitem{CF} A. Costa, M. Farber, 
\textit{Large random simplicial complexes, I.}
J. Topol. Anal. 8 (2016), no. 3, 399–429.

\bibitem{CLOT} O. Cornea, G. Lupton, J. Oprea, D. Tanr\'e, \textit{Lusternik-Schnirelmann category}. AMS, 2003.

\bibitem{Erd} P. Erd\"os. \textit{On an Elementary Proof of Some Asymptotic Formulas in the Theory of Partitions}. Ann. Math. 43(1942), 437-450.

\bibitem{EFM} C. Even-Zohar, M. Farber, L. Mead, \textit{Ample simplicial complexes}. European J. of Math., 8(2022), 1-32.

\bibitem{Far0} M. Farber, \textit{Topological Complexity of Motion Planning}. Journal of Discrete and Comput. Geom. 29(2003), 211–221.

\bibitem{Far2} M. Farber, \textit{Invitation to topological robotics}. Zurich Lectures in Advanced Mathematics. European Mathematical Society (EMS), Zürich, 2008.

\bibitem{Far} M. Farber, \textit{Large simplicial complexes: universality, randomness, and ampleness}. 
 J. Appl. Comput. Topol. 8 (2024), no. 6, 1551–1574.



\bibitem{FM} M. Farber, L. Mead, \textit{Random simplicial complexes in the medial regime}. Topology Appl. 272(2020), 107065, 22 pp.

\bibitem{FMN} M. Farber, L. Mead, T. Nowik, \textit{Random simplicial complexes, duality and the critical dimension.} J. Topol. Anal. 14 (2022), no. 1, pp. 1–31.


\bibitem{Ganea} T. Ganea, \textit{Lusternik - Schnirelmann category and strong category}, Illinois J. Math. 11 (1967), 417–427.
\bibitem{GLO} M. Grant, G. Lupton, and J. Oprea, \textit{Spaces of topological complexity one}. Homology Homotopy Appl. 15 (2013), no. 2, 73–81.

\bibitem{Gro} D.P. Grossman. \textit{An estimation of the category of Lusternik-Shnirelmann}. C.R. (Doklady) Acad. Sci. URSS (N.S.) 54, 1946.

\bibitem{Hatcher} A. Hatcher, \textit{Algebraic topology}, Cambridge University Press, 2002.

\bibitem{JLR} S. Janson, T. Luczak, A. Rucinski,  
Random graphs. Wiley-Interscience, New York, 2000

\bibitem{LM} N. Linial, R Meshulam, \textit{Homological connectivity of random 2-complexes}. Combinatorica 26 (2006), no. 4, 475–487. 

\bibitem{Shi} A.N. Shiryaev, \textit{Probability}, Springer, 1984.

\bibitem{S} E. Spanier, \textit{Algebraic Topology}, 1966. 

\bibitem{Wall} C.T.C. Wall, 
\textit{Finiteness conditions for CW-complexes,} Ann. of Math. (2) 81 (1965), 56–69.



\end{thebibliography}
\end{document}